\documentclass[11pt]{amsart}
\usepackage{amsmath,amssymb,amsthm,upref,graphicx,mathrsfs}
\usepackage[top=30mm, bottom=30mm, left=35mm, right=35mm]{geometry}
\usepackage{color,xcolor}
\usepackage[
  colorlinks=true,
  linkcolor=blue,
  citecolor=blue,
  urlcolor=blue]{hyperref}

\numberwithin{equation}{section}
\newtheorem{theorem}{Theorem}[section]
\newtheorem{proposition}[theorem]{Proposition}

\newtheorem{lemma}[theorem]{Lemma}
\newtheorem{remark}[theorem]{Remark}

\newtheorem{example}[theorem]{Example}

\newtheorem{question}[theorem]{Question}

\numberwithin{equation}{section}

\begin{document}
\title[Mean ergodic composition operators]{Mean ergodic composition operators on $H^\infty(\mathbb{B}_n)$}
\author{Hamzeh Keshavarzi }

\maketitle
\begin{abstract}
 In this paper, we study (uniformly) mean ergodic composition operators on $H^\infty(\mathbb{B}_n)$.
Under some additional assumptions, it is shown that mean ergodic operators have norm convergent iterates in $H^\infty(\mathbb{B}_n)$, and that they are always uniformly mean ergodic.
\end{abstract}

\section{Introduction}

Let $X$ be a Banach space.
For an operator $T$ on $X$, the iterates of $T$ are defined by $T^n:= T\circ \stackrel{(n)}{...} \circ T$. If $\sup_{n\in \mathbb{N}} \|T^n\|<\infty$, then $T$ is called power bounded. The Cesàro mean of the sequence formed by the iterates of $T$ is defined as:
$$M_j(T)=\dfrac{1}{j} \sum_{i=1}^j T^i.$$
We say that $T$ is mean ergodic if $M_j(T)$ converges to a bounded operator acting on $X$ for the strong operator topology. In addition, $T$ is called uniformly mean ergodic if $M_j(T)$ converges in the operator norm.

Composition operators on various spaces of holomorphic functions have been
studied in recent years. See the books by Cowen and MacCluer \cite{Cowen2}
and Shapiro \cite{Shapiro} for discussions of composition operators on classical spaces of holomorphic
functions. Several authors have studied the dynamical properties of composition operators \cite{arendt1, arendt2, beltran, bonet1, jorda, keshavarzi1, keshavarzi2}. In this paper, we focus on the
(uniform) mean ergodicity of these operators.

The study of mean ergodicity in the space of linear operators defined on a Banach space
goes back to von Neumann. In $1931$, he proved that if $H$ is a Hilbert space and $T$ is a unitary operator on $H$, then $T$ is mean ergodic. For a Grothendieck Banach space $X$ with Dunford-Pettis property (GDP space), Lotz \cite{Lotz} proved that $T\in L(X)$ which satisfies $\|T^n/n\|\rightarrow 0$ is mean ergodic if and only if it is uniformly mean ergodic. For the definition of GDP spaces see Lotz’s paper \cite[pages 208-209]{Lotz}.

Beltr\'{a}n-Meneua et. al \cite{beltran} characterized the (uniform) mean ergodicity of composition operators on $H^\infty(\mathbb{D})$. Since $H^\infty(\mathbb{D})$ is a GDP space, a composition operator acting on this space is mean ergodic if and only if it is uniformly mean ergodic. However, we do not know whether $H^\infty(\mathbb{B}_n)$ is a GDP space or not. Nevertheless, we will show that the mean ergodicity and the uniform mean ergodicity of composition operators on $H^\infty(\mathbb{B}_n)$ are equivalent in some cases (see Theorems \ref{t1}, \ref{t2} and Example \ref{ex1}).

If $X$ is a Banach space, the space of continuous linear operators
from $X$ to itself is denoted by $\mathcal{L}(X)$.
Furthermore, if $\mathcal{K}(X)\subseteq \mathcal{L}(X)$ denotes the set of compact operators on $X$, then the essential norm
$ \|T\|_e := \inf \{ \|T-K\| : \ K \in \mathcal{K}(X)\}$ indeed defines a norm on the
Calkin algebra $\mathcal{L}(X)/ \mathcal{K}(X)$.
An operator $T\in \mathcal{L}(X)$ is called quasi-compact if there exists an $n_0\in \mathbb{N}$ such that $ \|T^{n_0}\|_e<1$.

Let $\varphi$ be a holomorphic self-map of $\mathbb{D}$ with interior Denjoy-Wolff point $a\in \mathbb{D}$. In \cite{beltran} and \cite{jorda}, it has been shown that the following statements are equivalent:
\begin{itemize}
\item[(i)] $C_\varphi$ is mean ergodic on $H^\infty(\mathbb{D})$.
\item[(ii)] $C_\varphi$ is uniformly mean ergodic on $H^\infty(\mathbb{D})$.
\item[(iii)] $\| \varphi_j-a \|_\infty \rightarrow 0$, as $j\rightarrow \infty$.
\item[(iv)] $C_\varphi$ is quasi-compact on $H^\infty(\mathbb{D})$.
\end{itemize}
 Indeed, in the one-variable case, we have $\varphi_j\rightarrow a$, uniformly on the compact subsets of $\mathbb{D}$, that is, the holomorphic retraction associated with $\varphi$ is $\rho\equiv a$. In Theorem \ref{t2}, we show that by the assumption $\rho\equiv a$ in several variables, we have the same result. However, in section $4$, we present an example which shows that this theorem is practically different from the one-variable case.

The notation $A\lesssim B$ on a set
$S$ means that, independent of any $a\in S$, there exists some positive constant $C$ such that $A(a)\leq CB(a)$. Moreover,
we use the notation
$A\simeq B$ on $S$ to indicate that there are some
positive constants $C$ and $D$ such that $CB(a) \leq A(a)\leq DB(a)$ for each $a\in S$.

\section{Basic results}

If $X$ is a taut manifold, consider $Hol(X, X)$ as the set of all holomorphic functions $f: X\rightarrow X$ and let $Aut(X)$ be the set of all automorphisms of $X$. Let $\mathbb{C}$ be the complex plane.
Then, the unit ball in $\mathbb{C}^n$ is defined as
$\mathbb{B}_n=\{z\in \mathbb{C}^n: \ |z|<1 \}$ and the unit disk in $\mathbb{C}$ is denoted by $\mathbb{D}=\mathbb{B}_1$. It is well-known that $\mathbb{B}_n$ is a taut manifold. Let $H(\mathbb{B}_n)$ be the space of all holomorphic functions from $\mathbb{B}_n$ into $\mathbb{C}$ and let  $H^\infty(\mathbb{B}_n)$ be the subspace of all bounded functions in $H(\mathbb{B}_n)$.
If $\varphi$ is a holomorphic self-map of $\mathbb{B}_n$, then the composition operator $C_\varphi$ on $H(\mathbb{B}_n)$ induced by $\varphi$ is defined as $C_\varphi f=f\circ \varphi$.

An automorphism $\varphi$ of $\mathbb{B}_n$ is a unitary matrix of $\mathbb{C}^n$ if and only if $\varphi(0)=0$. Another class of automorphisms consists of symmetries of $\mathbb{B}_n$ which are also called involutive automorphisms or involutions. Thus, for any point $a\in \mathbb{B}_n\setminus \{0\}$, we define:
$$\varphi_a(z)=\dfrac{a-P_a(z)-s_aQ_a(z)}{1-\langle z,a\rangle}, \qquad z\in \mathbb{B}_n,$$
where $s_a=\sqrt{1-|a|^2}$, $P_a$ is the orthogonal projection from $\mathbb{C}^n$ onto the one-dimensional subspace $\langle a \rangle$ generated by $a$ and $Q_a$ is the orthogonal projection from $\mathbb{C}^n$ onto  $\mathbb{C}^n\ominus \langle a \rangle$. Clearly, $\varphi_a(0)=a$, $\varphi_a(a)=0$, and $\varphi_a \circ \varphi_a(z)=z$.
Every automorphism $\varphi$ of $\mathbb{B}_n$ has the following form:
$$\varphi = U\varphi_a= \varphi_b V,$$
where $U$ and $V$ are unitary transformations of $\mathbb{C}^n$ and $\varphi_a$ and $\varphi_b$ are involutions.

Let $\beta: \mathbb{B}_n \times \mathbb{B}_n \rightarrow [0,\infty)$ be the Bergman metric. We can formulate this metric as follows:
\begin{equation} \label{e0}
\beta(z,w)= \dfrac{1}{2} \log \dfrac{1+|\varphi_z(w)|}{1-|\varphi_z(w)|}, \qquad z,w\in \mathbb{B}_n.
\end{equation}
$\beta$ is invariant under automorphisms, that is,
$$\beta(\varphi(z),\varphi(w))=\beta(z,w),$$
for all $z,w\in \mathbb{B}_n$ and $\varphi\in Aut(\mathbb{B}_n)$. Moreover, if $\varphi:\mathbb{B}_n\rightarrow\mathbb{B}_n$ is just a holomorphic function, then from \cite[Proposition 2.3.1]{abate}, we have
\begin{equation}\label{e5}
\beta(\varphi(z),\varphi(w))\leq\beta(z,w), \qquad \forall z,w\in\mathbb{B}_n.
\end{equation}
 If $r>0$ and $z\in \mathbb{B}_n$, then we shall denote by $B(z,r)$ the Bergman ball centered at $z$ with radius $r$ and define it in the following way:
$$B(z,r)=\{w\in \mathbb{B}_n: \ \beta(z,w)<r\}.$$
If we define $\nu$ as the Lebesgue measure on $\mathbb{C}^n$, then for every $r>0$ we have
\begin{equation} \label{e1}
 (1-|z|^2)^{n+1} \simeq (1-|w|^2)^{n+1}\simeq \nu(B(z,r))\simeq \nu(B(w,r)),
\end{equation}
 and
\begin{equation} \label{e2}
 |1-\langle z,w\rangle|\simeq (1-|z|^2),
\end{equation}
for all $w\in \mathbb{B}_n$ and $z\in B(w,r)$.
For the proof of the above inequalities refer to Zhu’s book \cite[Lemmas 1.24 and 2.20]{zhu}.

 Now, we give some basic results from \cite{abate}. Let $X$ be a taut manifold, $z\in X$ and $\varphi \in Hol(X,X)$. We define  $d_z \varphi =\varphi^\prime(z)$ and $ sp \ d_z \varphi$ as the spectrum of the matrix $d_z \varphi$.
Let $T_z X$ be the tangent space at $X$ in $z$. If $z\in \mathbb{B}_n$, then it is well-known that $\mathbb{C}^n\cong T_z \mathbb{B}_n$.

\begin{theorem} \cite[Theorem 2.1.21]{abate} \label{t01}
Let $X$ be a relatively taut manifold fulfilling the identity principle and
take $\varphi \in Hol(X,X)$ with a fixed point $z_0 \in X$. Then:
\begin{itemize}
\item[(i)] the spectrum of $d_{z_0}\varphi$ is contained in $\overline{\mathbb{D}}$;
\item[(ii)] $|\det d_{z_0}\varphi | \leq 1$;
\item[(iii)] $d_{z_0}\varphi=id$ if and only if $\varphi$ is the identity;
\item[(iv)] $T_{z_0}X$ admits a $d_{z_0}\varphi $-invariant splitting $T_{z_0}X= L_N \oplus L_U$ such that the spectrum
of $d_{z_0}\varphi|_{L_N}$ is contained in $\mathbb{D}$, the spectrum of $d_{z_0}\varphi|_{L_U}$ is contained in $\partial \mathbb{D}$, and $d_{z_0}\varphi|_{L_U}$ is diagonalizable;
\item[(v)]  $| \det d_{z_0}\varphi| = 1$ if and only if $\varphi$ is an automorphism.
\end{itemize}
\end{theorem}

Let $X$ be a taut manifold.
A sequence of holomorphic maps $\{f_j\} \subset Hol(X,X)$
is compactly divergent if for every pair of compact sets $K_1,K_2 \subset X$, there
is $N \in \mathbb{N}$ such that $f_j(K_1) \cap K_2 = \emptyset$ for all $j\geq N$. If $X = \mathbb{B}_n$, we shall sometimes say
that $\{f_j\}$ diverges to infinity uniformly on compact sets.

  A holomorphic retraction
of $X$ (or with terminology borrowed from the semigroup theory, an idempotent of $Hol(X,X)$)
is a holomorphic map $\rho :X \rightarrow X$ such that $\rho^2 = \rho$. The image of a holomorphic retraction
is said to be a holomorphic retract of $X$. Consider $\varphi \in Hol(X,X)$ such that the sequence of its iterates $\{\varphi_{k}\}$ converges to a holomorphic
function $h:X \rightarrow X$. Then, the sequence $\{\varphi_{2k}\}$ tends to $h$ too. It follows that $h^2 = h$, that is, $h$ is a holomorphic retraction of $X$.

\begin{theorem} \cite[Theorem 2.1.29]{abate} \label{t02}
Let $X$ be a taut manifold and $\varphi \in Hol(X,X)$. Assume that the sequence
$\{\varphi_k\}$ of iterates of $\varphi$ is not compactly divergent. Then, there exist a unique submanifold $M$
of $X$ and a unique holomorphic retraction $\rho:X \rightarrow M$ such that every limit point $h \in Hol(X,X)$
of $\{\varphi_k\}$ is of the form $h = \gamma \circ \rho$,
where $\gamma$ is an automorphism of $M$. Moreover, even $\rho$ is a limit point of the sequence $\{\varphi_k\}$.
\end{theorem}

\begin{remark}
In \cite[Theorem 2.1.29]{abate} (the above theorem), it is not mentioned that $M$ and $\rho$ are unique. However, we can infer from its proof that they are unique.
\end{remark}

\begin{proposition} \cite[Proposition 2.2.14]{abate} \label{p03}
Let $\varphi \in Hol(\mathbb{B}_n,\mathbb{B}_n)$ be such that $\varphi(0) = 0$. Assume $\|\varphi(z_0)\| = \|z_0\|$ for some $z_0 \in \mathbb{B}_n$ with $z_0 \neq 0$ (or $\|d_0\varphi(v_0)\| = \|v_0\|$ for some $v_0 \in \mathbb{C}^n$ with $v_0\neq 0$).
Then, there is a linear subspace $V$ of $\mathbb{C}^n$ containing $z_0$ (or $v_0$) such that
$\varphi|_{V \cap \mathbb{B}_n} = d_0 \varphi|_{V \cap \mathbb{B}_n}$ is the restriction of a suitable unitary matrix.
\end{proposition}

\begin{theorem} \cite[Theorem 2.2.32]{abate} \label{t04}
 Let $\varphi \in Hol(\mathbb{B}_n,\mathbb{B}_n)$ have a fixed point  $z_0 \in \mathbb{B}_n$. Then,
the sequence of iterates $\{\varphi_k\}$ converges if and only if $ sp \ d_{z_0} \varphi \subset \mathbb{D} \cup \{1\}$. Furthermore, the sequence $\{\varphi_k\}$ converges to $z_0$ if and only if $ sp \ d_{z_0} \varphi \subset \mathbb{D}$.
\end{theorem}

\begin{theorem} \cite[Theorem 2.2.31]{abate} \label{t05}
Let $\varphi \in Hol(\mathbb{B}_n,\mathbb{B}_n)$ not have fixed points in $\mathbb{B}_n$. Then, the sequence of iterates of $\varphi$ converges to a point of $\partial \mathbb{B}_n$ uniformly on the compact subsets of $\mathbb{B}_n$.
\end{theorem}

\section{Main results}

In this section, we study the (uniformly) mean ergodic composition operators on $H^\infty(\mathbb{B}_n)$. The equivalence of mean ergodicity and uniform mean ergodicity in some cases will be given. First, we study the case in which $\varphi$ has an interior fixed point. Note that if $\varphi$ has an interior fixed point $a\in \mathbb{B}$, then $\psi:=\varphi_a \circ \varphi\circ \varphi_a$ is a holomorphic self-map of $\mathbb{B}_n$ that fixes the origin. By applying the chain rule and the fact that $\varphi_a \circ \varphi_a=id$ we obtain the following equations:
$$d_0 \psi=(d_a \varphi_a) (d_a \varphi) (d_0 \varphi_a), \ \ and \ \ (d_0 \varphi_a) (d_a \varphi_a)=(d_a\varphi_a) (d_0\varphi_a)=id.$$
According to these equations, $d_a \varphi$ and $d_0\psi$ are similar matrices. This similarity implies that $sp \ d_a \varphi=sp \ d_0\psi$. Hence,  without loss of generality, we can assume that $\varphi(0)=0$.

\begin{lemma} \label{l2}
Let $\varphi$ be a holomorphic self-map of the unit ball with  a fixed point $a\in \mathbb{B}_n$. If $C_\varphi$ is mean ergodic on $H^\infty(\mathbb{B}_n)$, then
$$ sp \ d_a \varphi \subset \mathbb{D} \cup \{\lambda \in \partial \mathbb{D}, \ which \ \lambda \ is \ a \ root \ of \ 1 \}.$$
\end{lemma}
\begin{proof}
According to the discussion before the lemma, without loss of generality, we can assume that $a=0$.
Let $\lambda \in  sp \ d_0 \varphi$. It is also assumed that it is not a root of $1$. If $v$ is the eigenvector of $d_0\varphi$ related to $\lambda$, then $sp \ d_0(\varphi\mid_{\langle v\rangle\cap \mathbb{B}_n})=\{\lambda\}$,  where $\langle v\rangle$ is the subspace spanned by $v$. Thus, by Theorem \ref{t01}, $\varphi\vert_{\langle v\rangle\cap \mathbb{B}_n } $ is an automorphism of $\langle v\rangle \cap \mathbb{B}_n$.
Since $\varphi(0)=0$, we can let $\varphi\vert_{\langle v\rangle\cap \mathbb{B}_n}:\mathbb{D} \rightarrow \mathbb{D}$ be the elliptic disk automorphism $\varphi\vert_{\langle v\rangle\cap \mathbb{B}_n}(z)=\lambda z$. Moreover, according to the assumption, $C_{ \varphi\vert_{\langle v\rangle \cap \mathbb{B}_n}}$ is mean ergodic on $H^\infty(\mathbb{D})$. This contradicts \cite[Theorem 2.2]{beltran}.
\end{proof}

The following theorem shows that when the index function is an automorphism of the unit ball with an interior fixed point, then the induced composition operator is mean ergodic if and only if it is uniformly mean ergodic on $H^\infty(\mathbb{B}_n)$.

\begin{theorem} \label{t1}
Let $\varphi$ be an automorphism of the unit ball with  a fixed point $a\in \mathbb{B}_n$. Then, the following statements are equivalent.
\begin{itemize}
\item[(i)]  $C_\varphi$ is mean ergodic on $H^\infty(\mathbb{B}_n)$.
\item[(ii)]  $C_\varphi$ is uniformly mean ergodic on $H^\infty(\mathbb{B}_n)$.
\item[(iii)] $ sp \ d_a \varphi \subset \{\lambda \in \partial \mathbb{D}, \ \lambda \ is \ a \ root \ of \ 1 \}.$
\end{itemize}
\end{theorem}

\begin{proof}
Again, without loss of generality, we assume that $a=0$.

(iii)$\Rightarrow$ (ii): Clearly, there is a positive integer $k$ such that $ sp \ d_0 \varphi_k =\{ 1 \}$. Since $\varphi$ is an automorphism and $\varphi(0)=0$, it is a unitary matrix. On the other hand, unitary matrices are diagonalizable. Thus, one can see that $\varphi_k=id$. Therefore,
$$M_j(C_\varphi )\longrightarrow \frac{1}{k}\sum_{i=0}^{k-1} C_{\varphi_i}.$$

(ii) $\Rightarrow$ (i) is obvious and (i) $\Rightarrow$ (iii) is a straightforward consequence of Theorem \ref{t01} (v) and Lemma \ref{l2}.
\end{proof}

Now, we are going to present the main result of this paper (Theorem \ref{t2}) for the interior fixed-point case. First, we need to give some auxiliary results.

\begin{lemma} \label{l3}
Let $\varphi$ be a holomorphic self-map of the unit ball with a fixed point in $\mathbb{B}$, $\rho$ be the holomorphic retraction associated with $\varphi$,  and $C_\varphi$ be mean ergodic. Then, there is a positive integer $k$ such that
$$\lim_{j\rightarrow \infty} \varphi_{kj}=\rho,$$
and $ \{\rho\circ \varphi_i, \ 0\leq i\leq k-1\}$ is the set of limit points of $\{\varphi_j\}$.
Moreover, $\rho\circ\varphi_{k+l}=\rho\circ \varphi_l$ for every positive integer $l$ and
$$\lim_{j\rightarrow\infty} M_j(C_\varphi )= \frac{1}{k}\sum_{i=0}^{k-1} C_{\rho\circ \varphi_i}.$$
\end{lemma}
\begin{proof}
Let $a\in \mathbb{B}$ be the fixed point of $\varphi$. Then, according to Lemma \ref{l2} there is a positive integer $k$ such that $sp \ d_a \varphi_k \subset \mathbb{D} \cup \{1\}$. Thus, Theorem \ref{t04} implies that there is a holomorphic function $h:\mathbb{B}_n\rightarrow \mathbb{B}_n$ so that
$$\lim_{j\rightarrow \infty} \varphi_{kj}=h.$$

We will show that $h=\rho$: The above limit implies that $h$ is the holomorphic retraction associated with $\varphi_k$. Hence, $h^2=h$.
It follows from Theorem \ref{t02} that there is an automorphism $\gamma$ on $M$ so that $h=\gamma\circ \rho$.
Thus, $h$ is an automorphism on $M$ such that $h^2=h$, that is, it is the identity function on $M$. Therefore, $h=\rho$.

Hence,
$$\lim_{j\rightarrow \infty} \varphi_{kj+i}=\rho\circ \varphi_i, \qquad 1\leq i\leq k-1.$$
That is, $\{\rho\circ \varphi_i, \ 0\leq i\leq k-1\}$ is a  subset of the limit points of $\{\varphi_j\}$. One can see that all the limit points of $\{\varphi_j\}$ are of the form $\rho\circ \varphi_i$.
Moreover, for $l\geq 0$, we have
$$\rho\circ\varphi_{k+l}=\lim_{j\rightarrow \infty} \varphi_{kj+k+l}=\lim_{j\rightarrow \infty} \varphi_{(k+1)j+l}=\rho\circ \varphi_{l}.$$
This completes the proof.
\end{proof}

For the proof of the following lemma see \cite[page 66]{zhu}.
\begin{lemma}\label{l4}
Let $0<p<\infty$ and $0<R<r$. Then, there is a positive constant $C$ which depends only on $r$ and $R$ such that
$$|\nabla f (z)|\leq C \Big( \int_{B(0,r)} |f(w)|^p d\nu (w) \Big)^\frac{1}{p}, \qquad z\in B(0,R), \ \ f\in H(\mathbb{B}_n).$$
\end{lemma}

\begin{lemma} \label{l1}
For any $r>0$, there exists a positive constant $C$ which depends only on $r$, such that
$$ | f(u)-f(v)| \leq C \beta(u,v) \|f\|_\infty,$$
for all $f\in H^\infty(\mathbb{B}_n)$ and $u,v\in \mathbb{B}_n$ so that $\beta(u,v)\leq r$.
\end{lemma}
\begin{proof}
Let $u=\varphi_v(w)$ and $g=f\circ \varphi_v$. Then,
$$ | f(u)-f(v)|= |g(w)-g(0)|= \Big| \int_0^1 \Big( \sum_{i=1}^n w_i \frac{\partial g}{\partial w_i} (tw) \Big) dt\Big|.$$
On the compact set $|w|\leq r$, the Bergman metric is equivalent to the Euclidean metric. Thus, there is a positive constant $C_r$ which depends only on r such that
$$| f(u)-f(v)|\leq C_r \beta(w,0) \sup \{ |\nabla g (z)| \ z\in B(0,r)\}.$$
Since the Bergman metric is automorphism-invariant, we have
$$| f(u)-f(v)|\leq C_r \beta(u,v) \sup \{ |\nabla f\circ \varphi_v (z)| \ z\in B(0,r)\}.$$
By Lemma \ref{l4}, there is a positive constant $D_r$ depending only on $r$ such that  the above value is less than or equal to
\begin{equation*}
D_r \beta(u,v) \Big( \int_{B(0,2r)} |f(\varphi_{v}(w))|^p d\nu(w)\Big)^{1/p}\leq D_r \nu(B(0,2r)^\frac{1}{p} \beta(u,v) \|f\|_\infty.
\end{equation*}
This completes the proof.
\end{proof}

The following theorem gives a sufficient condition for the uniform mean ergodicity of composition operators in the general interior fixed-point case.
\begin{theorem} \label{t3}
Let $\varphi$ be a holomorphic self-map of the unit ball with a fixed point in $\mathbb{B}_n$ and $\rho$ be the holomorphic retraction associated with $\varphi$. If there is a positive integer $k$ such that
$$\lim_{j\rightarrow \infty} \sup_{z\in \mathbb{B}_n} \beta(\varphi_{kj}(z),\rho(z))= 0,$$
 then $C_\varphi$ is uniformly mean ergodic.
\end{theorem}
\begin{proof}
Fix $r>0$. By the above assumption, there is a positive integer $N$ such that
$$\varphi_{kj}(z)\in B(\rho(z),r), \ \ j\geq N, \ and \ z\in \mathbb{B}_n.$$
Thus, from Lemma \ref{l1}, there is a positive constant $C_r$ so that
$$|f(\varphi_{kj}(z))-f(\rho(z))|^2 \leq C_r\beta(\varphi_{kj}(z),\rho(z)) \|f\|^2_\infty,$$
for all $j\geq N$ and $z\in \mathbb{B}_n$. Hence, $C_{\varphi_{kj}} \longrightarrow C_\rho$ in the operator norm of $H^\infty(\mathbb{B}_n)$. This implies that $C_{\varphi_{kj+i}}\rightarrow C_{\rho\circ\varphi_i}$ for all $0\leq i\leq k-1$. Therefore, one can see that
$$\lim_{j\rightarrow\infty} M_j(C_\varphi )= \frac{1}{k}\sum_{i=0}^{k-1} C_{\rho\circ\varphi_i},$$
that is, $C_\varphi$ is uniformly mean ergodic.
\end{proof}

\begin{remark}
Note that if $a\in \mathbb{B}_n$, then since the Bergman metric is invariant under automorphisms, we have
\begin{align*}
\sup_{z\in \mathbb{B}_n} \beta(\varphi_a\circ \varphi_{kj}\circ\varphi_a(z),\varphi_a\circ \rho\circ\varphi_a(z))&=\sup_{z\in \mathbb{B}_n} \beta(\varphi_{kj}\circ\varphi_a(z),\rho\circ\varphi_a(z))\\
&=\sup_{z\in \mathbb{B}_n} \beta(\varphi_{kj}(z),\rho(z))
\end{align*}
Thus, the sufficient condition given in Theorem \ref{t3} is equivalent to the condition
$$\lim_{j\rightarrow \infty} \sup_{z\in \mathbb{B}_n} \beta(\varphi_a\circ \varphi_{kj}\circ\varphi_a(z),\varphi_a\circ \rho\circ\varphi_a(z))= 0.$$
Compare this with condition (iii) in Theorem \ref{t2}.
\end{remark}

\begin{proposition} \label{p2}
Let $\varphi$ be a holomorphic self-map of the unit ball, $\varphi(0)=0$ and $\rho$ be the holomorphic retraction associated with $\varphi$. Then,
\begin{equation}\label{e10}
M=\rho(\mathbb{B}_n)=L_U\cap \mathbb{B}_n=\{z\in \mathbb{B}_n; \ |\varphi(z)|=|z|\},
\end{equation}
where $L_U$ is the vector space obtained in Theorem \ref{t01}.
\end{proposition}
\begin{proof}
It follows from \cite[Lemma 2.2.13 and Proposition 2.2.33]{abate} that $M$ coincides with $L_U\cap \mathbb{B}_n$.

The proof of the last equality is as follows: If $|\varphi(z)|=|z|$, then according to Proposition \ref{p03}, $\varphi\mid_{\langle z\rangle}$ is a unitary matrix. Hence, $sp \ d_0 \varphi\mid_{\langle z\rangle} \subset \partial \mathbb{D}$ which implies that $z\in L_U\cap \mathbb{B}_n$. Conversely, let $z\in L_U\cap \mathbb{B}_n$. Since $\varphi(0)=0$ and $sp \ d_0 \varphi\mid_{L_U} \subset \partial \mathbb{D}$, it follows that $\varphi\mid_{L_U}$ is a unitary matrix. Therefore, $|\varphi(z)|=|z|$.
\end{proof}

For $k>0$ and $\zeta\in \partial \mathbb{B}_n$, we define the ellipsoid
$$E(k,\zeta)=\{ z\in \mathbb{B}_n: \ |1-\langle z,\zeta\rangle|^2\leq k(1-|z|^2)\}.$$
Let $\rho$ be a holomorphic self-map of the unit ball and $0<\eta<1$. Put
$$L(\rho,\eta)=\{z\in \mathbb{B}_n, \ \ |z-\rho(z)|\geq\eta\}.$$
The following lemma is an extension of \cite[Lemma 12]{cowen1}.

\begin{lemma} \label{l6}
Let $\varphi$ be a holomorphic self-map of the unit ball, $\varphi(0)=0$, and $\rho$ be the holomorphic retraction associated with $\varphi$.  If $0<\eta<1$ so that $L(\rho,\eta)\neq \emptyset$, then there is some $A>1$ such that
$$\dfrac{1-|\varphi(z)|}{1-|z|}>A, \qquad \forall z\in L(\rho,\eta).$$
\end{lemma}
\begin{proof}
For every $z\in L(\rho,\eta)$, we have $|\varphi(z)|<|z|$. Indeed, if $z_0\in L(\rho,\eta)$ so that $|\varphi(z_0)|=|z_0|$, then it follows from \ref{e10} that $z_0\in M=\rho(\mathbb{B}_n)$. Since $\rho$ is the identity on $M$, we have $\rho(z_0)=z_0$ which contradicts the assumption  $z_0\in L(\rho,\eta)$.

Let
$$d(z)=\dfrac{1-|\varphi(z)|}{1-|z|}, \qquad \forall z\in \mathbb{D},$$
and
$$d(\zeta)=\liminf_{z\rightarrow\zeta}\dfrac{1-|\varphi(z)|}{1-|z|}, \qquad \forall \zeta\in \partial\mathbb{D}.$$
From Schwarz’s Lemma $d(\zeta)\geq 1$. We claim that $d(\zeta)>1$ for every $\zeta\in \partial L(\rho,\eta)\cap\partial\mathbb{B}_n$.
If $\zeta\in \partial L(\rho,\eta)\cap\partial\mathbb{B}_n$ so that $d(\zeta)=1$, then clearly $\varphi(\zeta)\in \partial\mathbb{B}_n$.

Consider $0<c<1$. It follows from Julia's lemma \cite[Lemma 2.77]{Cowen2} that the image of the ellipsoid $E(c/(1-c),\zeta)$ under $\varphi$
is in the ellipsoid $E(c/(1-c),\varphi(\zeta))$. Since $(1-2c)\zeta$ has the smallest modulus in $E(c/(1-c),\zeta)$, Schwarz's Lemma implies that  $|\varphi((1-2c)\zeta)|=|1-2c|$. Thus, $(1-2c)\zeta\in M$ and so
 $$\rho((1-2c)\zeta)=(1-2c)\zeta,$$
  for all $0<c<0$. This contradicts the assumption that $\zeta\in \partial L(\rho,\eta)\cap\partial\mathbb{B}_n$.

It is easy to show that all the limit points of $L(\rho,\eta)$ in $\mathbb{B}_n$ are in $L(\rho,\eta)$ itself.
Since $d$ is lower semicontinuous, $d$ takes its minimum value on the compact set
$$L(\rho,\eta)\cup (\partial L(\rho,\eta)\cap\partial\mathbb{B}_n).$$
In the first part, we proved $d(z)>1$ for all $z\in L(\rho,\eta)$. In the second part, it was proved that $d(\zeta)>1$ for all $\zeta\in \partial L(\rho,\eta)\cap\partial\mathbb{B}_n$. This leads to the desired result.
\end{proof}

E. Jord\'{a} and A. Rodr\'{i}guez-Arenas \cite[Proposition 1.2]{jorda} gave the following result which is a consequence of \cite[ Theorem 3.4, Corollaries (ii) and (iii)]{yosida2}. Recall the concept of the point spectrum $\sigma_p(T)$: the set of all $\lambda \in \mathbb{C}$ such that $T-\lambda I$ is not injective.

\begin{proposition} \label{p1}
The following conditions are equivalent for $T \in \mathcal{L}(X)$:
\begin{itemize}
\item[(a)] $\{T^j\}_j$ converges in $\mathcal{L}(X)$ to a finite rank projection $P$.
\item[(b)] $T$ is power bounded and quasi-compact. In addition, $\sigma_p(T) \cap \partial\mathbb{D} \subseteq \{1\}$.
\end{itemize}
\end{proposition}

From a careful read of \cite[Theorem 2]{bern}, \cite[Page 203]{hoffman} and \cite[Lemma 13]{cowen1}, one can obtain the following lemma.
\begin{lemma}\label{l5}
 Let $\{x_j\}$ be a sequence of points in $\mathbb{B}_n$. Consider the following conditions.
\begin{itemize}
\item[(i)] There is some $0<a<1$ such that for every positive integer $j$, we have
$$\dfrac{1-|x_{j+1}|}{1-|x_{j}|}<a.$$
\item[(ii)] There is some $\delta>0$ such that for every positive integer $k$, we have
$$\prod_{j, j\neq k} |\varphi_{x_k}(x_j)|\geq \delta.$$
\end{itemize}
Then, (i) implies (ii). Moreover, (ii) implies that there is an $M>0$ which depends only on $\delta$, and a sequence $\{f_{l}\}_{l=1}^\infty$ in $H^\infty(\mathbb{B}_n)$ such that:
\begin{itemize}
\item[(a)] $f_{l}(x_l)=1$ and $f_l(x_j)=0$ if $l\neq j$.
\item[(b)]  $\sum |f_l(z)|\leq M$ for all $z\in \mathbb{B}_n$.
\end{itemize}
\end{lemma}

Now, we are in a position to prove the main result.

\begin{theorem} \label{t2}
Let $\varphi$ be a holomorphic self-map of the unit ball with a fixed point $a\in \mathbb{B}_n$ so that $ sp \ d_a \varphi \subset \mathbb{D}$. Then, the following statements are equivalent.
\begin{itemize}
\item[(i)]  $C_\varphi$ is mean ergodic on $H^\infty(\mathbb{B}_n)$.
\item[(ii)]  $C_\varphi$ is uniformly mean ergodic on $H^\infty(\mathbb{B}_n)$.
\item[(iii)] $\| \varphi_j-a \|_\infty \rightarrow 0$, as $j\rightarrow \infty$.
\item[(iv)]  $C_\varphi$ is quasi-compact on $H^\infty(\mathbb{B}_n)$.
\end{itemize}
\end{theorem}
\begin{proof}
Without loss of generality, assume $a=0$.  Since $ sp \ d_0 \varphi \subset \mathbb{D}$, it follows from Theorem \ref{t04} that $\varphi_j\rightarrow 0$ uniformly on the compact subsets of $\mathbb{B}_n$. Thus, the holomorphic retraction associated with $\varphi$ is $\rho\equiv 0$.

(iii)$\Rightarrow$ (ii): Since $\rho\equiv 0$ and $\| \varphi_j \|_\infty \rightarrow 0$, we have
\begin{align*}
\sup_{z\in \mathbb{B}_n} \beta(\varphi_{j}(z),\rho(z))&=\sup_{z\in \mathbb{B}_n} \beta(\varphi_{j}(z),0)\\
&=\sup_{z\in \mathbb{B}_n} \dfrac{1}{2} \log \dfrac{1+|\varphi_{j}(z)|}{1-|\varphi_{j}(z)|}\rightarrow 0,
\end{align*}
as $j\rightarrow \infty$. Therefore, (ii) follows from Theorem \ref{t3}.

(ii)$\Rightarrow$ (i) is obvious.

(i)$\Rightarrow$ (iii):
Since $\varphi_j\rightarrow 0$ uniformly on the compact subsets of $\mathbb{B}_n$, the positive integer $k$ obtained in  Lemma \ref{l3} is $1$. Hence,
\begin{equation} \label{e4}
M_j(C_\varphi )\rightarrow  K_{0},
\end{equation}
 for the strong operator topology, where $K_0$ is the operator $K_0 f=f(0)$.
 Let (iii) do not hold.

 We claim that $\|\varphi_{j}\|_\infty=1$ for all $j$. Indeed, if there is some $j_0$ so that $\|\varphi_{j_0}\|<1$, then the closure of $\varphi_{j_0}(\mathbb{B}_n)$ is a compact subset of $\mathbb{B}_n$. Since $\varphi_j\rightarrow 0$ uniformly on the compact subsets of $\mathbb{B}_n$, we have
 $$\varphi_j=\varphi_{j-j_0}\circ \varphi_{j_0}\rightarrow 0$$
uniformly on $\mathbb{B}_n$. This contradicts the assumption that $(iii)$ does not hold. Thus, $\|\varphi_{j}\|_\infty=1$ for all $j$.\\

\textbf{Claim}.  For every $0<\varepsilon<1$, there exists a sequence $\{a_j\}$ in $\mathbb{B}_n$ and some $f$ in $H^\infty(\mathbb{B}_n)$ such that $|\varphi_{j}(a_j)|\geq\varepsilon$ for all $j$ and
$$f(0)=0, \ \ f(\varphi_l(a_j))=|\varphi_l(a_j)|^2,\qquad  \ 1\leq l\leq j, \ \forall j\in \mathbb{N}.$$
   Proof of the claim. Since $\rho\equiv 0$, we have $L(\rho,\varepsilon)=\{z\in \mathbb{B}_n:|z|\geq \varepsilon\}$. Thus, from Lemma \ref{l6}, there is a  constant $0<a<1$ such that
   \begin{equation} \label{e3}
   \dfrac{1-|z|}{1-|\varphi(z)|}< a, \ |z|\geq \varepsilon.
   \end{equation}

    Let $a_1$ in $\mathbb{B}_n$ be such that $|\varphi(a_1)|\geq \varepsilon$.
 Since $\|\varphi_{2}\|_\infty=1$, we can find an $a_2\in \mathbb{B}_n$ such that $|\varphi_{2}(a_2)|$ is large enough so that
$$|\varphi_{2}(a_2)|\geq \varepsilon, \ \ and \ \ \dfrac{1-|\varphi_{2}(a_2)|}{1-|\varphi(a_1)|}<a.$$
By using Schwarz's lemma, $|\varphi(a_2)|\geq |\varphi_2(a_2)|\geq \varepsilon$. Thus, from \ref{e3}, we have
$$\dfrac{1-|\varphi(a_2)|}{1-|\varphi_{2}(a_2)|}<a.$$
We continue this way and construct the following sequence
\begin{equation*}
\begin{matrix}
x_1=\varphi(a_1)\\
x_2=\varphi_{2}(a_2) & x_3=\varphi(a_2) \\
x_4=\varphi_{3}(a_3) & x_5=\varphi_2(a_3) & x_6=\varphi(a_3)\\
\vdots & \vdots & \vdots & \ddots
\end{matrix},
\end{equation*}
which satisfies condition (i) of Lemma \ref{l5}. Therefore, there is some positive constant $M$ and a sequence $\{f_{l,j}\}_{j,l=1}^{\infty,j}$ in $H^\infty(\mathbb{B}_n)$ such that
\begin{itemize}
\item[(a)] $f_{l,j}(\varphi_l(a_j))=1$, and $f_{l,j}(\varphi_r(a_s))=0$ whenever $l\neq r$ or $j\neq s$.
\item[(b)] $\sum_{j=1}^{\infty}\sum_{l=1}^{j} |f_{l,j}(z)|\leq M$, for all $z\in \mathbb{B}_n$.
\end{itemize}
We define the function
\begin{equation*}
f(z)=\sum_{j=1}^{\infty}\sum_{l=1}^{j}  \langle z,\varphi_l(a_j)\rangle f_{l,j}(z).
\end{equation*}
Hence, $f\in H^\infty(\mathbb{B}_n)$, $f(0)=0$, and
$$f(\varphi_l(a_j))=|\varphi_l(a_j)|^2, \ \ 1\leq j<\infty, \ 1\leq l\leq j.$$
 Therefore, the proof of the claim is complete.\\

Using Schwarz's lemma,
$$|\varphi_l(a_j)|^2 \geq|\varphi_{j}(a_j)|^2  \geq \varepsilon^2,$$
for all $j\geq 1$ and $0\leq l\leq j$. Therefore, it follows that
\begin{eqnarray*}
\|\dfrac{1}{j} \sum_{l=1}^{j} C_{\varphi_l} - K_{0} \| &\geq&
\dfrac{1}{\|f\|_\infty} \Big{\|}\dfrac{1}{j} \sum_{l=1}^{j} C_{\varphi_l} f-C_{0}  f \Big{\|}_\infty\\
&\geq&\dfrac{1}{\|f\|_\infty} \Big|\dfrac{1}{j} \sum_{l=1}^{j}  f(\varphi_l(a_j))-f(0)\Big|\\
&=& \dfrac{1}{\|f\|_\infty} . \dfrac{1}{j} \sum_{l=1}^{j} |\varphi_l(a_j)|^2 \geq \dfrac{\varepsilon^2}{\|f\|_\infty} .
\end{eqnarray*}
From this, we conclude that the sequence $\{M_{j}(C_\varphi )\}_{j=1}^\infty$ does not converge to $K_0$ for the strong operator topology which contradicts \ref{e4}. Thus, (iii) holds.\\

 (ii)$\Leftrightarrow$ (iv): It is easy to prove that $\lambda=1$ is an eigenvalue of $C_\varphi$ with constant functions as its associated eigenvectors. Furthermore, since the iterates of $\varphi$ are convergent to $0$, one can see that $\sigma_p(C_\varphi)\cap \partial\mathbb{D}=\{1\}$. Thus, the equivalence of  (ii) and (iv) follows from Proposition \ref{p1}.
\end{proof}

\begin{remark}
With the conditions in Theorem \ref{t2}, the sequence  $\{C_{\varphi_j}\}$ converges to the finite rank operator $K_0$. Hence, we can use Proposition \ref{p1} in Theorem \ref{t2}. However, in Theorem \ref{t3}, the sequence $\{C_{\varphi_{kj}}\}$  converges to $C_\rho$ which is not necessarily a finite rank operator.
\end{remark}

As the final result of this section, we show that if $\varphi$ has no interior fixed point, then  $C_\varphi$ is not uniformly mean ergodic. This argument is a refinement of the proof of \cite[Theorem $3.6$]{beltran}.

\begin{theorem} \label{t4}
Let the holomorphic function $\varphi:\mathbb{B}_n\rightarrow \mathbb{B}_n$ have no interior fixed point. Then, $C_\varphi$ is not uniformly mean ergodic on $H^\infty(\mathbb{B}_n)$.
\end{theorem}
\begin{proof}
From Theorem \ref{t05}, there is a $z_0 \in \partial \mathbb{B}_n$ such that the iterates of $\varphi$ converge to $z_0$ uniformly on the compact subsets of $\mathbb{B}_n$. Hence, $M_j(C_\varphi)\rightarrow K_{z_0}$ on
$$A(\mathbb{B}_n)=H(\mathbb{B}_n)\cap \{f:\overline{\mathbb{B}_n}\rightarrow \mathbb{C}, \ continuous\},$$
where $K_{z_0}(f)=f(z_0)$ on $A(\mathbb{B}_n)$.
Consider $g(z)=\langle \frac{z_0+z}{2}, z_0\rangle\in A(\mathbb{B}_n)$. Then, $g(z_0)=1$ and $|g(z)|<1$ for all $z\in \mathbb{B}_n$. Fix $j\in \mathbb{N}$ and take $r_j>0$ so that
$$\{\varphi_l(0)\}_{l=1}^j\cap \{z: \ |z-z_0|<r_j\}=\emptyset.$$
We know that
$$s_j=\sup \{|g(z)|; \ z\in \mathbb{B}_n \setminus \{z: \ |z-z_0|<r_j\}\}<1.$$
Let $k_j\in \mathbb{N}$ be such that $s_j^{k_j}<1/2$ and consider the functions $g_j=g^{k_j}$. Note that $\{g_j\}$ is a bounded sequence in $A(\mathbb{B}_n)$ and $g_j(z_0)=1$ for every $j$. Hence, we have
$$\|\dfrac{1}{j} \sum_{l=1}^{j} C_{\varphi_l} - K_{z_0} \|_{A(\mathbb{B}_n)} \geq \dfrac{1}{\|g_j\|_\infty} \Big|g_j(z_0)-\dfrac{1}{j} \sum_{l=1}^j  g_j(\varphi_l(0))\Big|\geq \dfrac{1}{2 \sup_{j\in \mathbb{N}} \|g_j\|_\infty}, \ \ \forall j.$$
Therefore, $C_\varphi$ is not uniformly mean ergodic on $H^\infty(\mathbb{B}_n)$.
\end{proof}

We add the following questions for the future investigations of mean ergodic composition operators:
\begin{question}
Does the inverse of Theorem \ref{t3} hold?
\end{question}

\begin{question}
Is it true that composition operators, induced by holomorphic functions that  have no interior fixed points, are not mean ergodic (extending Theorem \ref{t4})?
\end{question}

\section{\textbf{Example}}

Let $A$ be an $n\times n$-complex matrix with not necessarily distinct eigenvalues $\{\lambda_1,...,\lambda_n\}$ and  singular values $\delta_1\geq ...\geq \delta_n$. It is well-known that $A(0)=0$, $d_0 A=A$, and
$$\sup_{z\in\mathbb{B}_n} |Az|=\delta_1.$$
If we consider $A$ as a holomorphic function on $\mathbb{B}_n$, then $A$ acts from $\mathbb{B}_n$ to $\mathbb{B}_n$ if and only if $\delta_1\leq 1$. Moreover, $\|A\|_\infty =\delta_1$.

 If $|\lambda_i|<1$ for all  $1\leq i\leq n$, then it follows from \cite[Page 617]{meyer} that $\|A^j\|_\infty\rightarrow 0$ as $j\rightarrow \infty$.
Hence, Theorem \ref{t2} implies that $C_A$ is uniformly mean ergodic.

Now, assume
$$ sp \ d_0 A =sp \ A \subset \mathbb{D} \cup \{\lambda \in \partial \mathbb{D}; \ \lambda \ is \ a \ root \ of \ 1 \}.$$
Thus, there is some positive integer $k$ such that
 $$sp A^k \subset \mathbb{D} \cup \{ 1 \}.$$
Again, \cite[Page 630]{meyer} implies that $A^{kj}$ converges. Let $V$ be an invertible matrix so that $V^{-1} A V$ is the Jordan canonical form of $A$. It is easy to show that the holomorphic retraction associated with $V^{-1}  A^k V$ is $V^{-1} \rho V=0_s \oplus I_{n-s}$ and
\begin{equation*}
V^{-1} A^{kj} V=
\begin{bmatrix}
\lambda_1^{kj} & ? & \cdots & ? \\
 0 & \lambda_2^{kj} & \cdots & ?\\
 \vdots &  & \ddots & \vdots \\
 0 & 0 & \cdots & \lambda_s^{kj}\\
\end{bmatrix}
\oplus I_{n-s}=B^{kj}\oplus I_{n-s},
\end{equation*}
where $0_{s}$ is the $s\times s$-zero matrix, $I_{n-s}$ is the $(n-s)\times (n-s)$-identity matrix, $\lambda_i\in \mathbb{D}$ for $1\leq i\leq s$, and
\begin{equation*}
B=
\begin{bmatrix}
\lambda_1 & ? & \cdots & ? \\
 0 & \lambda_2 & \cdots & ?\\
 \vdots &  & \ddots & \vdots \\
 0 & 0 & \cdots & \lambda_s\\
\end{bmatrix}
\end{equation*}
Again, from \cite[Page 617]{meyer}, we have $\|B^j\|_\infty\rightarrow 0$ as $j\rightarrow \infty$.

Let $X$ be a taut manifold and $k_X$ be the Kobayashi metric on $X$. For the definition of Kobayashi metric see \cite[Page 158]{abate}. From \cite[Corollary 2.3.6]{abate}, the Kobayashi metric and the Bergman metric on $\mathbb{B}_n$ are coincide.

\begin{lemma} \label{l7}
Let $X$ be a taut manifold in $\mathbb{C}^n$ and $z,w\in X$ so that $z=(z_1,...,z_n)$ and $w=(0,...,0,z_{s+1},...,z_n)$. Then,
$$k_X(z,w)\leq \omega(|(z_1,...,z_s)|,0),$$
where $\omega$ is the Poincar\'{e} metric on $\mathbb{D}$.
\end{lemma}
\begin{proof}
Let $\varphi:\mathbb{D}\rightarrow X$ be defined as
$$\varphi(\zeta)=\zeta \dfrac{(z_1,...,z_s,0,...,0)}{|(z_1,...,z_s)|}+(0,...,0,z_{s+1},...,z_n).$$
Then, $\varphi(0)=0$ and $\varphi(|(z_1,...,z_s)|)=z$. This gives the desired result.
\end{proof}

If $k$ is the Kobayashi metric on $V^{-1}(\mathbb{B}_n)$, then from \cite[Proposition 2.3.1]{abate} and Lemma \ref{l7}, we obtain
\begin{align*}
\sup_{z\in \mathbb{B}_n}\beta(A^{kj}(.),\rho (.))&=\sup_{z\in \mathbb{B}_n}k_{\mathbb{B}_n}(A^{kj}(.),\rho (.))=\sup_{z\in V^{-1} \mathbb{B}_n}k_{\mathbb{B}_n}(A^{kj}V(.),\rho V (.))\\
&\leq\sup_{z\in V^{-1} \mathbb{B}_n} \kappa_{V^{-1}\mathbb{B}_n} (V^{-1} A^{kj} V(.),V^{-1}\rho V(.))\\
&\leq \sup_{z\in V^{-1} \mathbb{B}_n} \omega(|B^{kj}\oplus 0_{n-s}(z)|,0)\\
&=\dfrac{1}{2} \sup_{z\in V^{-1} \mathbb{B}_n} \tanh^{-1} (|B^{kj}\oplus 0_{n-s}(z)|)\rightarrow 0.
\end{align*}
as $j\rightarrow\infty$. Thus, from Theorem \ref{t3}, $C_A$ is uniformly mean ergodic.
Therefore, if
$$sp \ A \subset \mathbb{D} \cup \{\lambda \in \partial \mathbb{D}; \ \lambda \ is \ a \ root \ of \ 1 \},$$
 then $C_A$ is (uniformly) mean ergodic. The converse is obtained from Lemma \ref{l2}. The results of this section are given in the following example:

\begin{example} \label{ex1}
Let $A$ be an $n\times n$-complex matrix with $\|A\|_\infty\leq 1$. Then, the following statements are equivalent.
\begin{itemize}
\item[(i)]  $C_A$ is mean ergodic on $H^\infty(\mathbb{B}_n)$.
\item[(ii)]  $C_A$ is uniformly mean ergodic on $H^\infty(\mathbb{B}_n)$.
\item[(iii)] $ sp \ A \subset \mathbb{D} \cup \{\lambda \in \partial \mathbb{D}; \ \lambda \ is \ a \ root \ of \ 1 \}.$
\end{itemize}
\end{example}

\vspace*{5mm}

\textbf{Acknowledgments}. This work was supported by the Iran National Science Foundation (INSF) [project number: 4000186]. The author would like to thank professor E. Jordá for his invaluable advice, careful reading of the paper, and all his help which highly improved the paper. Finally, I am grateful to the referee for her/his helpful comments.

\vspace*{0.5cm}

 Hamzeh Keshavarzi

E-mail: Hamzehkeshavarzi67@gmail.com

Department of Mathematics, College of Sciences,
Shiraz University, Shiraz, Iran

\end{document}